\documentclass[11pt]{amsart}
\pdfoutput=1
\usepackage{amsmath,amsthm,amssymb,hyperref,fancyhdr,mathtools}
\usepackage{derivative}
\usepackage[thinc]{esdiff}
\usepackage{color}
\usepackage{comment}
\usepackage{enumitem}

\theoremstyle{plain}
\newtheorem{theorem}{Theorem}
\newtheorem{lemma}[theorem]{Lemma}
\newtheorem{proposition}[theorem]{Proposition}

\newtheorem{definition}[theorem]{Definition}

\numberwithin{theorem}{section}
\numberwithin{equation}{section}

\newcommand{\ip}[2]{\ensuremath{\left\langle#1,#2\right\rangle}}

\newcommand{\R}{\mathbb{R}}
\newcommand{\C}{\mathbb{C}}

\newcommand{\N}{\mathbb{N}}

\newcommand{\LL}{\mathcal{L}}

\newcommand{\FF}{\mathcal{F}}

\newcommand{\WW}{\mathcal{W}}

\newcommand{\B}{\mathbb{B}}
\newcommand{\D}{\mathbb{D}}

\newcommand{\fL}{\mathfrak{L}}

\newcommand{\cA}{\mathcal{A}}

\newcommand{\cS}{\mathcal{S}}
\newcommand{\cF}{\mathcal{F}}

\newcommand{\vm}{\textbf{\textit{m}}}

\newcommand{\cbuf}{C_{b,u}(\mathbb{C}^n)}
\newcommand{\cbuop}{C_{b,u}(\FF^2)}
\newcommand{\bddf}{L^\infty(\C^n)}
\newcommand{\bddop}{\LL(\FF^2)}

\newcommand{\domop}{\WW^{2,\infty}(\FF^2)}
\newcommand{\domoptrace}{\WW^{2,1}(\FF^2)}
\newcommand{\domopp}{\WW^{2,p}(\FF^2)}
\newcommand{\domopq}{\WW^{2,q}(\FF^2)}

\newcommand{\F}{\cF^2(\C^n)}

\newcommand{\cbuC}{\mathcal{C}_{b,u}(\C^n)}
\newcommand{\cbuR}{\mathcal{C}_{b,u}(\R)}

\newcommand{\domf}{\WW^{2,\infty}(\C^n)}
\newcommand{\domfone}{\WW^{2,1}(\C^n)}
\newcommand{\domfp}{\WW^{2,p}(\C^n)}

\newcommand{\bdd}{\fL(\cF^2)}
\newcommand{\rad}{\fL_{rad}(\cF^2)}

\newcommand{\toepalgrad}{\mathfrak{T}_{rad}(L^\infty)}
\newcommand{\toepalg}{\mathfrak{T}(L^\infty)}
\newcommand{\traceop}{\cS^1(\cF^2)}

\newcommand{\tconv}[1]{\varphi_t \ast #1}
\newcommand{\conv}[1]{\varphi \ast #1}

\newcommand{\cbu}{\mathcal{C}_{b,u}(\N_0,\rho)}

\newcommand{\dtwo}{d_\Delta}

\newcommand{\lipext}{f_{\sigma}^+}
\newcommand{\lipextR}{f_{\sigma}}

\newcommand{\actop}[2]{L_{#1}(#2)}
\newcommand{\actf}[2]{\ell_{#1}(#2)}

\newcommand{\Tber}[1]{\varphi_t\ast #1}

\newcommand{\Tr}{\mathrm{Tr}}
\newcommand{\one}{1}

\begin{document}


\title[The Laplacian of an operator]{The Laplacian of an operator and the radial Toeplitz algebra}

\author{Vishwa Dewage }
\author{Mishko Mitkovski }
\thanks{M.~M. was supported by NSF grant DMS-2000236.}

\address{Department of Mathematics, Clemson University\\
South Carolina, SC 29634, USA\\
E-Mail Dewage:vdewage@clemson.edu
E-Mail Mitkovski: mmitkov@clemson.edu}


\begin{abstract}
    Using tools from quantum harmonic analysis, we show that the domain of the Laplacian of an operator is dense in the Toeplitz algebra over the Fock space $\mathcal{F}^2(\mathbb{C}^n)$. As an application, we provide a simplified treatment of the Gelfand theory of the radial Toeplitz algebra.
\end{abstract}

\maketitle


\section{Introduction}

In a series of papers Su\'arez explores the "Laplacian of an operator" and its connections to Toeplitz operators on the Bergman space $\cA^2(\mathbb{D})$ over the unit disc $\D$ in $\C$ \cite{S04,S05,S08}. Several of his results were generalized to unit ball $\B^n$ in $\C^n$ in \cite{BHV14b, BHV14}. His work yields particularly strong results, some of which were later used in \cite{GMV13} to explore the Gelfand theory of the commutative radial Toeplitz algebra on $\cA^2(\mathbb{D})$ (see \cite{GKV03}), identifying it as an algebra of uniformly continuous sequences. 

The commutativity of Toeplitz algebras and their Gelfand theory has been a popular topic in operator theory \cite{BR22, DOQ18, DO22, EV16, FR23}.
As for the Fock space $\F$, the radial Toeplitz algebra is commutative (see \cite{GV02}) and in \cite{EM16}, using approximation theory, Maximenko and Esmeral prove that it is isometrically isomorphic to  $C^*$-algebra $\cbu$ of all bounded sequences that are uniformly continuous with respect to the square-root metric.

By now, Werner's quantum harmonic analysis (QHA) \cite{W84} is proven to be an effective tool in several fields of mathematics, including time-frequency analysis, mathematical physics, and operator theory \cite{BBLS22, FH23, F19, FR23, KLSW12, KL21, LS20}. In this article, we employ QHA to investigate the Laplacian of an operator on the Fock space $\F$. In particular, as one of our main results, we note that the Toeplitz algebra over $\F$ is the norm closure of the domain of this Laplacian. We then utilize this result to revisit the Gelfand theory of the radial Toeplitz algebra on $\F$, providing a short natural proof of the fact that it is isometrically isomorphic to the sequence algebra $\cbu$.

The article is organized as follows. In section \ref{sec:prelim}, we set the framework of the problem and present an overview of quantum harmonic analysis concepts we use. In section \ref{sec:Laplacian}, we define the Laplacian of an operator and discuss some of its properties. In section \ref{sec:density}, we also prove that the domain of the Laplacian $\domop$ is a dense subset of the Toeplitz algebra $\toepalg$. In section \ref{sec:Gelfand}, we present our new proof of the Gelfand theory of the radial Toeplitz algebra.


\section{Preliminaries} \label{sec:prelim}
Let $d\lambda(z)=e^{-\pi|z|^2}dz$ ($dz$ denotes the Lebesgue measure on $\C^n\simeq \R^{2n}$) be the standard Gaussian measure on $\C^n$. We consider the Bargmann-Fock space $\F$, the closed subspace of $L^2(\C^n, d\lambda)$ consisting of all entire functions that are square integrable with respect to the Gaussian measure $d\lambda$. It is well-known that $\F$ is a reproducing kernel Hilbert space with reproducing kernel 
$K_z(w)=e^{\pi\Bar{z}w}$ and normalized reproducing kernel $$k_z(w)=\frac{K_z(w)}{\|K_z\|}=e^{\pi\Bar{z}w-\frac{\pi}{2}|z|^2}.$$ 
Here we use the notation $\Bar{z}w=\Bar{z}_1w_1+\cdots+\Bar{z_n}w_n$. The orthogonal projection $P:L^2(\C,d\lambda)\to \FF^2(\C^n)$ is given by
\[P f(z)=\ip{f}{K_z}=\int_{\C^n} f(w)e^{\alpha z\Bar{w}}\ d\lambda(w).\]
The monomials $$e_{\vm}(z)=\sqrt{\frac{\pi^\vm}{\vm!}}z^{\vm},\ \ \ \vm\in\N_0^n$$
form an orthonormal basis for $\F$.

We denote the algebra of all bounded operators on $\F$ by $\bdd$. For $1\leq p<\infty$, denote by $\cS^p(\cF^2)$ the space of all Schatten-$p$-class operators on $\F$. We also make the identification $\cS^\infty(\cF^2)=\bdd$ when necessary.

\subsection{Toeplitz operators} For $a:\C^n\to \C$, the Toeplitz operator  $T_a:\FF^2(\C) \to \FF^2(\C)$ with symbol $a$ is formally given by $$T_a f(z)=P(af)(z)=\langle af,K_z\rangle; \ \ \forall \ z\in \C.$$
The \textit{Toeplitz algebra}, denoted by $\toepalg$, is the closed sub-algebra of all bounded operators $\bddop$ generated by the set of all Toeplitz operators with bounded symbols. Since adjoints of Toeplitz operators are also Toeplitz operators, $\toepalg$ is a noncommutative $C^*$-algebra.\\

\subsection{Heat kernel} The heat kernel plays a crucial role in the analysis on the Fock space \cite{BC94, BCI10}.
Let $\varphi_t$ denote the Gaussian approximate identity given by
\[\varphi_t(z)=\frac{1}{t^n}e^{-\frac{\pi|z|^2}{t}}, \ \ t>0.\] 
The function $\varphi_1$ is denoted simply by $\varphi$.
Recall that the Laplacian $\Delta$ acting on twice differentiable functions on $\C^n$ is given by
$$\Delta =\partial_1\Bar{\partial}_1+\cdots+\partial_n\Bar{\partial}_n$$
where $\partial_i$ and $\Bar{\partial}_i$ denotes $i^{th}$ partial derivatives with respect to $z_i$ and $\Bar{z}_i$ respectively.
Then $\varphi_t$ can be viewed as a solution of the heat equation $\Delta u=\pi\frac{\partial}{\partial t} u$, with initial condition $u(0)=\delta_0$. 

For $a\in \bddf$, the convolution $\tconv{a}$ is a solution to the heat equation:
$$\Delta(\tconv{a})=\pi \odv{}{t}(\tconv{a}).$$

\subsection{Quantum harmonic analysis (QHA)}
Here, we discuss convolutions between functions and operators as well as the convolutions between two operators via a unitary representation. We omit proofs and point to \cite{F19, H23, LS18, W84} for details. First, recall the usual translation of a function $f$ on $\C^n$ by $z\in \C^n$ given by 
$$\actf{z}a(w)=a(w-z), \ \forall w\in \C^n$$
and the convolution of two functions $a,\psi:\C^n \to \C$ is given formally by
$$(\psi\ast a)(z):=\int a(z-w)\ \psi(w)dw,\ \ \ z\in \C^n.$$
Note that the above convolution is commutative. 
Denote by $\cbuf$, the $C^*$-algebra of functions in $\bddf$ that are uniformly continuous w.r.t. the Euclidean metric, i.e.
$$\cbuf=\{f\in \bddf \mid z\mapsto\actf{z}{f} \text{ is continuous w.r.t. } \|\cdot\|_\infty \}.$$

Assume $\psi\in L^1(\C^n,dz)$.  Then if $a\in \bddf$, we have $\psi\ast a\in \cbuf$. Also if $a\in L^p(\C^n,dz)$, we have $\psi\ast a\in L^p(\C^n,dz)$ and the convolution satisfies the Young's inequality. \\

To discuss QHA, recall that the Weyl unitary operators acting on $\F$, given by 
$$W_zf(w)=k_z(w)f(w-z);\ \ w\in \C^n, \ f\in \F.$$
The map $z\mapsto W_z$ is an irreducible square-integrable projective unitary representation of $\C^n$.
We define the translation of an operator $S\in \bdd$ by $z\in \C^n$ to be
$$\actop{z}{S}=W_zS W_z^*, \ \ S\in \bdd.$$ 
The translation map $z\to \actop{z}{S}$ is continuous in SOT. The convolution of a function $\psi:\C^n\to \C$ and an operator $S\in \bdd$ is defined formally in the weak sense by
$$S\ast \psi :=\psi\ast S:=\int   \actop{z} S \ \psi(z)dz.$$

\begin{definition}
    We say an operator $S\in \bdd$ is uniformly continuous if the map $z\mapsto \actop{z}{S},\ \C^n\to \bdd$ is continuous. We denote the set of all bounded uniformly continuous operators by $\cbuop$, i.e.
    $$\cbuop=\{S\in \bdd \mid z\mapsto \actop{z}{S} \text{ is continuous w.r.t. operator norm.} \}$$
\end{definition}

Then $\cbuop$ is a $C^*$-subalgebra of $\bdd$ containing the set of all compact operators.

The convolution $\psi\ast S\in \cbuop$ in the following cases where is it well-defined, and in addition satisfies the QHA
Young's inequalities: For $1\leq p\leq \infty$
\begin{enumerate}[label=(\roman*)]
    \item If $\psi\in L^1(\C^n, dz)$ and $S\in \cS^p(\cF^2)$ then $\psi\ast S\in \cS^p(\cF^2)$ and
    $\|\psi\ast S\|_p\leq \|\psi\|_1\|S\|_p$.
    \item If $\psi\in L^p(\C^n,dz)$ $S\in \cS^1(\cF^2)$ then  $\psi\ast S\in \cS^p(\cF^2)$ and there is $C_\pi>0$ s.t.
    $\|\psi\ast S\|_p\leq C_\pi^{1-\frac{1}{p}}\|\psi\|_1\|S\|_p$.
    \item  In addition, $\psi \ast S\in \cbuop$ in both of the above cases.
\end{enumerate}

We have the following basic properties of operator translations and convolutions for the above cases where the convolution is well-defined. Let $\psi,\psi_1,\psi_2:\C^n\to \C$ and $S\in\bdd$ s.t. the following convolutions are defined. Then the translation commutes with convolution and the convolution is associative:
    \begin{enumerate}[label=(\roman*)]
        \item $\actop{z}{\psi\ast S}=(\actf{z} {\psi})\ast S =\psi\ast (\actop{z}{S}) $.
        \item $\psi_1\ast (\psi_2\ast S)=(\psi_1\ast \psi_2)\ast S=(\psi_2\ast \psi_1)\ast S$.
    \end{enumerate} 
Let $\{\psi_t\}$ be an approximate identity in $L^1(\C^n,dz)$. Then for $1\leq p <\infty$ and $S\in\cS^p(\cF^2)$, we have that $\psi_t \ast S\to S$ in Schatten-$p$ norm. Also, 
$$\cbuop=\{S\in \bdd \mid \psi_t \ast S\to S \text{ in operator norm}\}.$$

We define the convolution between two operators $S\in \bdd$ and $T\in \traceop$, denoted $S\ast T$, to be the function on $\C^n$ given by
\[S\ast T(z):=\Tr(SL_z(UTU))\]
where the self-adjoint operator $U$ on $\F$ is given by
$$(Uf)(z)=f(-z),\ \ \ z\in \C^n,\ f\in \F.$$
Then $T\ast S\in \cbuC$ and $\|T\ast S\|_\infty\leq \|T\|_1\|S\|$.

Furthermore, if $1\leq p\leq \infty$ and $S\in  \cS^p(\cF^2)$, then $S\ast T\in L^p(\C^n, dz)$ and satisfies the following QHA Young's inequality:
$$\|S\ast T\|_p\leq C_\pi^{\frac{1}{p}}\|S\|_p\|T\|_1.$$
For $A,C\in\bdd$, $B\in \traceop$, the following identities hold
    \begin{enumerate}[label=(\roman*)]
        \item $A\ast B= B\ast A$.
        \item $(A\ast B)\ast C=A\ast (B\ast C)$ if one of $A$ and $C$ is trace-class.
    \end{enumerate}  
 Also if $A,B\in \traceop$ and $a\in \bddf$, we have
$$a\ast (A\ast B)=(a\ast A)\ast B.$$

\subsection{Toeplitz operators as convolutions} Let $\Phi\in \bdd$ be the rank one operator $\Phi:=k_0\otimes \overline{k_0}=1\otimes\overline{1}$. 
One sees that for $a\in \bddf$, we have
\[T_a=a\ast \Phi.\] 
due to the fact that $L_z\Phi=k_z\otimes k_z$. This simple observation links Toeplitz operators and quantum harmonic analysis. As a result of the basic properties of convolutions, we get: for $a\in L^\infty(G)$ and $\psi\in L^1(G)$, $\actop{z}{T_a}=T_{\actf{z} a}$ and  $\psi\ast T_a= T_{\psi \ast a}$.

Since $a\ast \Phi\in \cbuop$ we immediately see that the set of all Toeplitz operators with bounded symbols, and thus the whole Toeplitz algebra $\toepalg$ is contained in $\cbuop$. In ~\cite{F19}, Fulsche proves that these two $C^*$-algebras coincide.

\subsection{Berezin transform as a convolution} The Berezin transform of a function $a:\C^n\to \C$ is formally defined to be
$$B(a)(z)=\ip{ak_z}{k_z}, \ z\in\C^n,$$
where $k_z$'s are the normalized reproducing kernels in $\F$.
Likewise, the Berezin transform of the operator $S$, is given by
$$B(S)(z)=\ip{Sk_z}{k_z}, \ z\in\C^n.$$
We see that $B(T_a)=B(a).$

One easily verifies that $\varphi=\Phi\ast \Phi$ and
\[B(a)=\varphi\ast a, \text{ and } B(S)=\Phi\ast S.\] Therefore, both $B(a)$ and $B(S)$ are uniformly continuous functions. In addition,
\[\varphi\ast S=(\Phi\ast \Phi)\ast S=(\Phi\ast S)\ast \Phi= T_{B(S)},\] for every bounded operator $S$.

\section{Laplacian of an operator}\label{sec:Laplacian}

In \cite{S08}, Su\'arez discusses the notion of Laplacian of an operator using the Berezin transform for the Bergman space over the unit ball $\cA^2(\D)$. Here we define the "Laplacian of an operator" on $\F$ which is the analog of the Laplacian that was discussed in \cite{S08} for $\cA^2(\D)$, and later in section \ref{sec:Gelfand}, we make use of this Laplacian to provide a short proof of a result on the Gelfand theory of the radial Toeplitz algebra. We note that this Laplacian aligns well with QHA much like its classical counterpart does with classical harmonic analysis: we have an associated heat equation and the Laplacian behaves well with convolutions. In addition, this respects Toeplitz quantization:
$$\Delta T_a=T_{\Delta a}$$
when $\Delta a$ exists and is bounded.

For $1\leq p\leq \infty$, consider the linear space of functions 
$$\domfp=\{f:\C^n\to \C\mid \Delta f \text{ exists and } \Delta f\in L^p(\C^n,dz)\}.$$


\begin{definition}
    Let $1\leq p \leq \infty$. We define the linear subspace $\domopp$ of $\bdd$ as follows:
    $$\domopp=\{S\in \bdd\mid \Delta B(S)=B(T) \text{ for some } T\in \cS^p(\cF^2) \}.$$
\end{definition}

Note that when $1\leq p\leq q\leq \infty$, we have $\cS^p(\cF^2)\subset \cS^q(\cF^2)$ and thus $\domopp\subset \domopq$ and particularly we have $\domopp \subset \domop$.  

\begin{definition}
    Assume $S\in \domop$, i.e. $\Delta B(S)=B(T),$ for some $T\in \bdd$. The operator $T$ is called "the Laplacian of $S$" and we write $$\Delta S=T.$$ 
\end{definition}

We call the linear space $\domop$ the "\textit{domain of the Laplacian}".


\subsection{Properties of operator Laplacian.}

Here we establish a few basic properties of the Laplacian and provide a few examples of operators in the domain of the Laplacian $\domop$ that includes Toeplitz operators.

The following proposition addresses the behavior of the Laplacian with convolutions.

\begin{proposition}\label{prop:domain lemma}
    Let $1\leq p\leq \infty$, $a\in L^p(\C^n,dz)$, $\psi\in L^1(\C^n, dz)$ and $S\in \bdd$. Then 
    \begin{enumerate}
        \item If $\psi\in \domfone$ then $\psi\ast a\in \domfp$ and  $\Delta (\psi\ast a)=(\Delta\psi)\ast a$.
        \item If $a\in \domfp$ then $\psi\ast a\in \domopp$ and $\Delta (\psi\ast a)=\psi\ast (\Delta a)$.
        \item If $\psi\in \domfone$ then $\psi\ast S\in \domop$ and $\Delta (\psi\ast S)=(\Delta\psi)\ast S$.
        \item If $S\in \domop$ then $\psi\ast S\in \domop$ and $\Delta (\psi\ast S)=\psi\ast \Delta S$.
        \item If $a\in \domfp$ and $S\in \traceop$, then $a\ast S\in \domopp$ and $\Delta (a\ast S)=(\Delta a)\ast S$.
        \item If $S\in \domoptrace$, then $a\ast S\in \domopp$ and $\Delta (a\ast S)=a\ast \Delta S$.
    \end{enumerate}
\end{proposition}

\begin{proof}
    The first two statements are classical statements. Proofs of others are similar. We prove only the statement (4). Note that since $S\in \domop$,
    $$\Delta(S\ast \Phi)=(\Delta S)\ast \Phi$$
    and hence $S\ast \Phi \in \domf$. Thus by associativity of the convolutions and statement (2),
    \begin{align*}
        \Delta B(\psi\ast S)&=\Delta((\psi\ast S)\ast \Phi)\\
        &= \Delta(\psi\ast (S\ast \Phi))\\
        &= \psi\ast \Delta(S\ast \Phi)\\
        &= \psi\ast ((\Delta S)\ast \Phi)\\
        &= B(\psi\ast \Delta S),
    \end{align*}  
    proving the lemma.
\end{proof}

Now we compute $\Delta \Phi$ for the rank one operator $\Phi$.
\begin{proposition}\label{prop:Phi in dom}
    The rank one operator $\Phi$ is in $\domoptrace$ and  $$\Delta \Phi= \pi\sum_{|\vm|=1}(E_\vm-\Phi). $$
\end{proposition}

\begin{proof}
     Note that $B(\Phi)(z)=\varphi(z)=e^{-\pi|z|^2}.$ Thus we compute
    \begin{align*}
        \Delta B(\Phi)(z) &= \pi(\pi|z|^2-n)e^{-\pi|z|^2}\\
        &= \pi\Big(\sum_{|\vm|=1} B(E_\vm)(z)-B(\Phi)(z)  \Big)\\
        &=B\Big( \pi\sum_{|\vm|=1}(E_\vm-\Phi)\Big)
    \end{align*}
    Also $E_\vm,\Phi\in \traceop$. Therefore $\Phi\in \domoptrace$ and the claim is true.
\end{proof}

Now we show that all Toeplitz operators are in $\domop$.

\begin{proposition}\label{prop:Toeplitz in dom}
    Let $1\leq p\leq \infty$. Then the set of all Toeplitz operators with symbols in $L^p(\C^n, dz)$ is contained in $\domop$. Moreover, if $a\in L^p(\C^n, dz)$,
    $$\Delta T_a= a\ast \Delta \Phi.$$
    In addition, if $a\in \domfp$, 
    we have
    $$\Delta T_a= T_{\Delta a}$$
\end{proposition}

\begin{proof}
    Since a Toeplitz operator $T_a$ can be written as $T_a=a\ast \Phi$, the first claim is a direct consequence of Proposition \ref{prop:Phi in dom} and statement (4) in Proposition \ref{prop:domain lemma}. The second claim follows from statement (5) of Proposition \ref{prop:domain lemma}.
\end{proof}

\subsection{Laplacian and the convolution between two operators}
Here we investigate how the Laplacian behaves with the convolution between two operators. We also discuss several equivalent conditions for an operator to be in $\domop$. 

In the following propositiom, we prove that $B(S)\in \domf$ for any operator $S\in \bdd$.
\begin{proposition}\label{prop:Laplacian of B(S)}
    Let $1\leq p\leq \infty$ and let $S\in \cS^p(\cF^2)$. Then $B(S)\in \domfp\cap \domf$ and
    $$\Delta B(S)= \Delta (S\ast \Phi)=S\ast \Delta \Phi.$$
\end{proposition}

\begin{proof}
    The first claim follows from a direct computation.
    For the second part, recall that
    $$\Delta \Phi= \pi\sum_{|\vm|=1}(E_\vm-\Phi). $$
    Hence
    \begin{align*}
        (S\ast \Delta \Phi)(z)&= \pi \sum_{|\vm|=1} \Tr(S\actop{z}{E_\vm-\Phi})\\
        &=\pi \sum_{|\vm|=1}\big( \Tr(SW_ze_\vm\otimes W_ze_\vm)-\Tr(SW_z\one\otimes W_z\one)\big)\\
        &= \pi \sum_{|\vm|=1}\big( \ip{SW_ze_\vm}{W_ze_\vm}-\ip{SW_z\one}{W_z\one}\big)
    \end{align*}
    Note that if $\vm\in \N_0^n$ is s.t. $\vm$ has a nonzero entry, which is a $1$, only on its $i^{th}$ position. Then $|\vm|=1$ and for $w\in \C^n$, 
    $$(W_ze_\vm)(w)=\sqrt{\pi}k_z(w)(w_i-z_i).$$
    Also $W_z\one=k_z$. Therefore
    \begin{align*}
        (S\ast \Delta \Phi)(z)&= \pi e^{-\pi|z|^2} \sum_{i=1}^n\Big(\pi\ip{S\big((w_i-z_i)K_z\big)}{(w_i-z_i)K_z}-\ip{SK_z}{K_z}\Big).
    \end{align*}
    Then we verify the following by computation:
    \begin{equation}\label{eq:Laplacian of B(S)}
        \Delta B(S)(z)= \pi e^{-\pi|z|^2} \sum_{i=1}^n\Big(\pi\ip{S\big((w_i-z_i)K_z\big)}{(w_i-z_i)K_z}-\ip{SK_z}{K_z}\Big),
    \end{equation}
    Proving $\Delta B(S)=S\ast \Delta \Phi$.
    Lastly, note that by QHA Young's inequality,
    $$\|S\ast \Delta \Phi\|_\infty\leq \|S\|\|\Delta \Phi\|_1$$
    and
    $$\|S\ast \Delta \Phi\|_p\leq C_\pi^{\frac{1}{p}}\|S\|_p\|\Delta \Phi\|_1$$
    
    and hence $B(S)\in \domf$.
\end{proof}

Given below is an equivalent condition for $S\in \domop$. 
\begin{proposition}\label{prop:dom equivalence}
    Let $S \in \bdd$. Then $S\in \domop$ iff $\Delta (\varphi\ast S)=\varphi\ast T$ for some $T\in \bdd$. Moreover, when the statement is true, $T=\Delta S$.
\end{proposition}

\begin{proof}
    Recall that $\varphi=\Phi\ast \Phi$ and the associativity of the convolution of three operators. Thus we have
    \begin{align*}
        \Delta(\varphi\ast S)&=\Delta ((\Phi\ast \Phi)\ast S)\\
        &= \Delta ((S\ast \Phi)\ast \Phi)\\
        &= \Delta(S\ast \Phi)\ast \Phi
    \end{align*}
    where the last line follows from statement (5) in Proposition \ref{prop:domain lemma} as $S\ast \Phi\in \domf$ by Proposition \ref{prop:Laplacian of B(S)}.
    Also for $T\in \bdd$, 
    $$\varphi\ast T= (\Phi\ast \Phi)\ast T=(T\ast \Phi)\ast \Phi.$$
    Hence 
    \begin{align*}
        &\Delta(\varphi\ast S)=\varphi\ast T\\ 
        &\iff  \Delta(S\ast \Phi)\ast \Phi= (T\ast \Phi)\ast \Phi\\
        &\iff \Delta(S\ast \Phi)=T\ast \Phi
    \end{align*}
    where the last line follows from the injectivity of the map $a\mapsto a\ast \Phi$.
\end{proof}

Next we discuss another characterization of the Laplacian.
\begin{proposition}\label{prop:Berezi Laplacian}
    Let $1\leq p \leq \infty$. Then
    \begin{enumerate}
        \item If there is $T\in \cS^p(\cF^2)$ s.t.
        $$\lim_{t\to 0^+}\frac{\varphi_t\ast S-S}{t}=\frac{1}{\pi}T$$
        in $\cS^p(\cF^2)$ then $S\in \domopp$ and $\Delta S=T$.
        \item If $p<\infty$ and $S\in \domopp$ then 
        $$\lim_{t\to 0^+}\frac{\varphi_t\ast S-S}{t}=\frac{1}{\pi}\Delta S$$
        w.r.t. the Schatten-$p$ norm $\|\cdot\|_\infty$.
        \item If $p=\infty$ and $\Delta S\in \cbuop$ then 
        $$\lim_{t\to 0^+}\frac{\varphi_t\ast S-S}{t}=\frac{1}{\pi}\Delta S$$
        w.r.t. the operator norm $\|\cdot\|$.
    \end{enumerate}
\end{proposition}
\begin{proof}
    To prove (1), note that $B(S)\in \domf\cap \cbuf$ by Proposition \ref{prop:Laplacian of B(S)}. Now by the classical theory we have
    \begin{align*}
        \Delta B(S)&= \pi \lim_{t\to 0^+}\frac{\varphi_t\ast B(S)-B(S)}{t}\\
        &= \pi \lim_{t\to 0^+} \Big( \frac{\varphi_t\ast S-S}{t} \Big)\ast \Phi\\
        &= T \ast \Phi
    \end{align*}
    by the QHA Young's inequality and the assumption. In addition, $$\|T \ast \Phi\|\leq C_\pi^{\frac{1}{p}}\|T\|_p\|\Phi\|_1.$$ Thus $S\in\domopp$, proving (1).
    
    The proofs of (2) and (3) are similar. We prove (3). Since $B(S)\in \domf$ by the heat equation and the statement (5) of \ref{prop:domain lemma}
    \begin{align*}
        \odv{}{t}(\varphi_t \ast B(S))
        &= \frac{1}{\pi}\Delta (\varphi_t\ast B(S))\\
        &= \frac{1}{\pi}\varphi_t\ast \Delta B(S)\\
        &= \frac{1}{\pi}\varphi_t\ast B(\Delta S).
    \end{align*}
    Hence by the fundamental theorem of calculus
    $$\varphi_t\ast B(S)-B(S)=\frac{1}{\pi}\int_0^t \varphi_s\ast B(\Delta S) \ ds,$$
    i.e.,
    $$(\varphi_t\ast S-S)\ast \Phi=\Bigg(\frac{1}{\pi} \int_0^t \varphi_s\ast\Delta S\ ds \Bigg)\ast \Phi$$
    Now since the Berezin transform is injective
    \begin{equation}\label{eq:heat integral}
        \varphi_t\ast S-S=\frac{1}{\pi}\int_0^t \varphi_s\ast\Delta S\ ds.
    \end{equation}
    Then we have
    \begin{align*}
        \Big\|\frac{\varphi_t\ast S-S}{t} -\Delta S\Big\| \leq \frac{1}{\pi t} \int_0^t \|\varphi_s\ast\Delta S-\frac{1}{\pi}\Delta S\| \ ds. 
    \end{align*}
    Then RHS tends to $0$ as $t\to 0^+$ because $\Delta S\in \cbuop$ and $\varphi_t\ast\Delta S\to S$ in operator norm as $t\to 0^+$. This proves 
    $$\lim_{t\to 0^+}\frac{\varphi_t\ast S-S}{t}=\frac{1}{\pi}\Delta S.$$   
    The proof of (2) is similar and uses the convergence $\varphi_t\ast \Delta S\to \Delta S$ in Schatten-$p$ norm.
\end{proof}

Lastly, we investigate the behaviour of the Laplacian with operator-operator convolution.

\begin{proposition}
    Let $1\leq p\leq \infty$, $S\in \cS^p(\cF^2)$ and $A\in \traceop$.Then
    \begin{enumerate}
        \item If $A\in \domoptrace$ then $S\ast A\in \domfp$ and  $\Delta (S\ast A)=S\ast \Delta A$.
        \item If $p<\infty$ and $S\in \domopp$ then $S\ast A \in \domf$ and  $\Delta (S\ast A)=(\Delta S)\ast  A$.
        \item If $S\in \domop$ and $\Delta S\in \cbuop$ then $S\ast A \in \domf$ and  $\Delta (S\ast A)=(\Delta S)\ast  A$.
    \end{enumerate}
\end{proposition}
\begin{proof}
    Proof of (1): 
    Note that 
    \begin{align*}
        \frac{\varphi_t\ast (S\ast A)-S\ast A}{t}&= S\ast \Big( \frac{\varphi_t\ast A-A}{t} \Big)
    \end{align*}
    Now since $$\lim_{t\to 0^+}\frac{\varphi_t\ast A-A}{t}=\frac{1}{\pi}\Delta A$$
    w.r.t. the trace norm, we have
    $$\lim_{t\to 0^+}\frac{\varphi_t\ast (S\ast A)-S\ast A}{t}= \frac{1}{\pi}(S\ast \Delta A)$$
    by the QHA Young's inequality.
    Therefore $S\ast A\in \domf$ with $\Delta (S\ast A)=S\ast \Delta A$.
    The proofs of (2) and (3) are similar.
\end{proof}

\section{Density of \texorpdfstring{$\domop$}{the domain} in \texorpdfstring{$\toepalg$}{the Toeplitz algebra}}\label{sec:density}
Here we show that $\domop$ is a dense subset of the Toeplitz algebra $\toepalg$.
In \cite{F19}, Fulsche discusses correspondence theory for Fock spaces and gives several characterizations of the Toeplitz algebra $\toepalg$:
\begin{align*}
    \toepalg&=\cbuop\\
    &=\{S\in \bdd \mid \varphi_t\ast S\to S \text{ in operator norm}\}\\
    &=\overline{\{T_a\mid a\in \cbuf\}}.
\end{align*}

We use the above characterizations of $\toepalg$ in our analysis. First, we examine the following operator version of the heat equation which is then used to prove the inclusion of $\domop$ in $\cbuop$.

\subsection{The heat equation}
If $\{T_t\}_{t>0}$ is a parametrized family of operators in $\bdd$, we define  the derivative of $T_t$ at $t_0\in (0,1]$ to be 
$$\odv{}{t} T_t =\lim_{t\to t_0} \frac{T_t-T_{t_0}}{t-t_0}$$
w.r.t. operator norm if it exists. When $\odv{}{t} T_t$ exists we also have the following:
$$\ip{\odv{}{t} T_t f_1}{f_2}=\odv{}{t}\ip{T_tf_1}{f_2}$$
for $f_1,f_2\in \F$.

\begin{lemma}\label{lem:derivative}
    Let $t\in (0,1]$ and define $T_t=\varphi_t\ast S$. Then $$\odv{}{t}T_t=\Big(\frac{\partial}{\partial t} \varphi_t\Big)\ast S$$
\end{lemma}

\begin{proof}
    Note that
    \begin{align*}
        \odv{}{t} T_t&= \lim_{s\to 0^+} \bigg( \Big(\frac{\varphi_t-\varphi_{t+s}}{s}\Big)\ast S\bigg)
    \end{align*}
    Also by the dominated convergence theorem
    $$\lim_{s\to 0^+}\frac{\varphi_t-\varphi_{t+s}}{s}= \frac{\partial}{\partial t} \varphi_t$$
    in $L^1(\C^n, dz)$.
    Now by the QHA Young's inequality, we have
    $$\odv{}{t} T_t=\Big(\frac{\partial}{\partial t} \varphi_t\Big)\ast S$$
    in operator norm, as required.
    
\end{proof}

\begin{proposition}\label{prop:heat equation}
    Let $S\in \bdd$ and $t\in(0,1]$. 
    Moreover, $\varphi_t\ast S$ satisfies the operator heat equation $\Delta T_t=\pi \odv{}{t} T_t$ for $t>0$.
\end{proposition}

\begin{proof}
    Since $$\Delta \varphi_t (z)=\pi \frac{\partial}{\partial t} \varphi_t (z)=\Big(\frac{-n}{t^{n+1}}+\frac{\pi|z|^2}{t^{n+1}}\Big)\varphi_t(z),$$ 
    we have $\varphi_t\in \domfone$ and hence the claim follows from statement (i) in Proposition \ref{prop:domain lemma}. The second claim follows from Lemma \ref{lem:derivative} as $\varphi_t\ast S\in \domop$:
    $$\Delta(\varphi_t\ast S)=(\Delta \varphi_t)\ast S=\pi\Big(\frac{\partial}{\partial t} \varphi_t\Big) \ast S.$$
\end{proof}

\subsection{Density of \texorpdfstring{$\domop$}{the domain} in \texorpdfstring{$\toepalg$}{the Toeplitz algebra}}

\begin{proposition}\label{prop:inclusion of domain}
     We have the inclusion:
     $$\domop\subset \cbuop.$$
\end{proposition}

\begin{proof}
    Let $S\in \domop$. It is enough to show that
     $\Tber{S} \to S$ norm as $t\to 0^+$.
     As seen in equation \ref{eq:heat integral} in the proof of Proposition \ref{prop:Berezi Laplacian}, we have
     $$\varphi_t\ast S=S+\frac{1}{\pi}\int_0^t \varphi_s\ast\Delta S\ ds.$$
    Now since $\|\varphi_s\ast\Delta S\|\leq \|\Delta S\|$, we have that $\varphi_t\ast S$ converges in operator norm to some operator $R\in \bdd$ as $t\to 0^+$.

It follows that $\conv{(\Tber{S})}$ also converges to $\conv{R}$ by the continuity of $S\mapsto \conv{S}$. Also $\conv{(\Tber{S})}=\Tber{(\conv{S})}\to \conv{S}$ as $\conv{S}\in \cbuop$.
 Hence $\conv{S}=\conv{R}$. Also the map $S\mapsto \conv{S}$ is injective because it is the composition of two injective maps: $\varphi\ast S=(S\ast \Phi)\ast \Phi$ and the Berezin transform and the map $a\mapsto T_a$ are injective.
Therefore $S=R$ as required.
\end{proof}

\begin{theorem}
    We have the following equality:
    $$\cbuop=\overline{\domop}.$$
\end{theorem}

\begin{proof}
    The inclusion $\domop\subset \cbuop$ was already proved in Proposition \ref{prop:inclusion of domain}. Now the density of $\domop$ in $\cbuop$ follows immediately from the fact that $\domop$ contains the set of all Toeplitz operators with bounded symbols (Proposition \ref{prop:Toeplitz in dom}) which is a dense subset of $\cbuop$ by correspondence theory in \cite{F19}.
\end{proof}


\section{Gelfand theory of the radial Toeplitz algebra}\label{sec:Gelfand}

In this section, we revisit the Gelfand theory of the well-known radial Toeplitz algebra via the operator Laplacian. 

\subsection{Radial Toeplitz operators}
Let $U_n$ denote the group of $n\times n$ unitary matrices. A function $f$ on $\C^n$ is radial if $f(A^{-1}z)=f(z)$ for all $z\in \C^n, A\in U_n$. A radial function depends only on $|z|$ and thus we may write $f(z)=f(|z|)$ with a slight abuse of notation. Let $\pi$ be the representation of $U_n$ acting on $\F$ given by
$$(\pi(A)f)(z)=f(A^{-1}z),\ \ \ z\in \C^n, A\in U_n$$
An operator $S\in \bdd$ is said to be radial if $S$ intertwines with $\pi$. The algebra of all radial operators is denoted by $\rad$. It is known that a Toeplitz operator $T_a$ is radial iff its symbol $a$ is a radial function. Moreover, radial operators on $\F$ are diagonalizable and the eigenvalue sequence of a radial Toeplitz operator $T_a$ is given by the formula:
$$\gamma_{n,a}(m)=\frac{1}{(n-1+m)!}\int_0^\infty a(\sqrt{r})r^{m+n-1}e^{-r}\ dr.$$
Note that we have used a normalization of the Gaussian measure that is different from \cite{GV02} but the proofs are similar.
 The radial Toeplitz algebra $ \toepalgrad$ is the $C^*$-algebra generated by the set of all radial Toeplitz operators. It was shown in \cite{GV02} the $\toepalgrad$ is commutative. Esmeral and Maximenko discusses the Gelfand theory of  $\toepalgrad$ in \cite{EM16} and proved that it is isometrically isomorphic to the $C^*$-algebra $\cbu$ of all bounded uniformly continuous sequences on $\N_0$ that are uniformly continuous with respect to the square-root metric $\rho:\N_0\times \N_0\to \R_+$ given by
$$\rho(m_1,m_2)=|\sqrt{m_1}-\sqrt{m_2}|,\ \ \ m_1,m_2\in \N_0.$$
Here we provide a concise proof of this fact. We start by noting that radial operators in $\domop$ are dense in the radial Toeplitz algebra $\toepalgrad$. \\

\begin{proposition}\label{prop:radial correspondence}
    The set of all radial operators in the domain of the Laplacian $\domop\cap \rad$ is dense in $\toepalgrad$, i.e.,
    $$\toepalgrad=\overline{\domop\cap \rad}=\overline{\{T_a\mid a\in \bddf \text{ is radial }\}}$$
\end{proposition}

\begin{proof}
    Note that by Proposition \ref{prop:Toeplitz in dom}, and since $\domop\subset \toepalg$, we have
    \begin{align*}
        \{T_a\mid a\in \bddf \text{ is radial }\}&\subset\domop\cap \rad\\
        &\subset \toepalg\cap \rad
    \end{align*}
    But also from Theorem 4.2 in \cite{DM23}, we get 
    $$\toepalgrad=\toepalg\cap \rad=\overline{\{T_a\mid a\in \bddf \text{ is radial }\}}$$
    proving the claim
    $\toepalgrad=\overline{\domop\cap \rad}.$
\end{proof}

Furthermore, due to the fact that $\cbu$ is invariant under shifts and the fact that $\gamma_{n,a}$ is a shift of $\gamma_{1,a}$, the problem reduces to the case $n=1$. We assume $n=1$ for now and verify the proof in Proposition \ref{prop:toep algebra n}.
Let  
$$\Gamma:\rad\to \ell^\infty(\N_0),\ S\mapsto \lambda_S$$ 
be the spectral map of radial operators. 
Denote the image of $\domop\cap\rad$ under the spectral map $\Gamma$ by $\dtwo$, i.e.
$$\dtwo=\Gamma \Big(\domop\cap\rad\Big).$$
Since $\domop\cap\rad$ is dense in the radial Toeplitz algebra $\toepalgrad$ by Proposition \ref{prop:radial correspondence}, the radial Toeplitz algebra is isomorphic to $\overline{\dtwo}$. Hence it is enough to show that $$\overline{\dtwo}=\cbu.$$
As noted in \cite{EM16}, $\cbu$ is a $C^*$-algebra containing eigenvalue multi-sequences of all quasi-radial Toeplitz operators. Thus by the minimality of the radial Toeplitz algebra we have that
 $$\overline{\dtwo}\subset\cbu$$

To prove the other inclusion, we start by giving an explicit description for $\dtwo$ in terms of the difference operator.


\subsection{A characterization of the sequence space \texorpdfstring{$\dtwo$}{d}}

Here we characterize the sequence space $\dtwo$ using the difference operator. First, we prove the following useful lemma.
\begin{lemma}\label{lem:SOT_conv_Delta_B}
    Let $S_k, S\in \bdd$ s.t. $S_k\to S$ in strong operator topology. Then the following convergences hold pointwise:
    \begin{enumerate}
        \item $B(S)(z)=\lim_k B(S_k)(z)$ for all $z\in \C^n$.
        \item $\Delta B(S)(z)=\lim_{k} \Delta B(S_k)(z)$
    for all $z\in \C^n$.
    \end{enumerate}
\end{lemma}

\begin{proof} 
    The first claim holds by the weak convergence of $S_k$ to $S$. From the proof of Proposition \ref{prop:Laplacian of B(S)}, we get 
    \begin{align*}
        \Delta B(S)&(z)=\pi e^{-\pi|z|^2} \sum_{i=1}^n\Big(\pi\ip{S\big((w_i-z_i)K_z\big)}{(w_i-z_i)K_z}-\ip{SK_z}{K_z}\Big).
    \end{align*}
    Hence the Lemma is true as $S_k\to S$ in weak operator topology.
\end{proof}

Recall that the monomial basis $\{e_{m}\}_{m\in \N_0}$ for $\cF^2(\C)$. Denote by $E_{m}$ the projection onto the span of $e_{m}$. Then for $f\in \cF^2(\C)$, 
$$E_{m}f=\ip{f}{e_{m}}e_{m}$$
If $\{\lambda_m\}$ are the eigenvalues of a radial operator $S$, then $$S=\sum_{m\in \N_0} \lambda_{m} E_{m}$$
holds in strong operator topology. Hence by Lemma \ref{lem:SOT_conv_Delta_B}, the following equalities hold pointwise:
\begin{align}\label{eq:ptwise_Berezin}
    B(S)(z)= \sum_{m\in \N_0} \lambda_{m}  B(E_{m})(z);\ \ \ z\in \C
\end{align}
and
\begin{align}\label{eq:ptwise_Berezin_Laplacian}
    \Delta B(S)(z)= \sum_{m\in \N_0} \lambda_{m}  \Delta B(E_{m})(z);\ \ \ z\in \C.
\end{align}

Now we compute  $\Delta E_{m}$.
\begin{lemma}
    The projections $E_{m}$ are in $\domoptrace$. Moreover,
    $$\Delta E_{m}=\pi\big( m E_{m-1}-(2m+1)E_{m}+(m+1)E_{m+1}\big).$$
\end{lemma}

\begin{proof}
    Notice that $$B(E_{m})(z)=\ip{E_{m}k_z}{k_z}=|\ip{e_{m}}{k_z}|^2=e^{-\pi|z|^2}|\ip{e_{m}}{K_z}|^2$$
Hence by reproducing formula
$$B(E_{m})(z)=|e_{m}(z)|^2e^{-\pi|z|^2}=\frac{\pi^m|z^{m}|^2e^{-\pi|z|^2}}{m!}.$$
From a direct computation, we get that
 $$\Delta B(E_{m})(z)=\pi\Big( m B(E_{m-1})(z)-(2m+1)B(E_{m})(z)+(m+1)B(E_{m+1})(z)\Big).$$
\end{proof}

Recall that the difference operator acting on a sequence $\{x_m\}$ is given by
$$\Delta x_m= x_{m+1}-x_{m}.$$
In the next lemma, we find a condition for a radial operator $S$ to be in $\domop$.

\begin{proposition}\label{prop:dtwo}
    Suppose $S\in\rad$ is given by
    $$S=\sum_{m\in \N_0} \lambda_{m} E_{m}$$
    where the convergence holds in strong operator topology and the
    bounded sequence $\{\lambda_{m}\}$ denotes the eigenvalues of $S$. 
    Define the sequence $\{\mu_m\}$ by
    $$\mu_m=  m\Delta^2\lambda_{m-1}+\Delta\lambda_{m}$$
    where we use the notation $m\Delta^2\mu_m=0$ when $m=0$.
    Then $S\in \domop$, iff $\{\mu_{m}\}$ is a bounded sequence. Moreover if $S\in \domop$,
    $$\Delta S=\sum_{m\in \N_0}\mu_{m} E_{m}.$$
\end{proposition}

\begin{proof}
    By equation \ref{eq:ptwise_Berezin_Laplacian}, we have
    $$\Delta B(S)(z)= \sum_{m\in \N_0} \lambda_{m} \Delta B(E_{m})(z);\ \ \ z\in \C^n.$$
    Hence
    \begin{align*}
        \Delta B(S)(z) &=  \sum_{m\in \N_0} \lambda_{m} (m B(E_{m-1})(z)-(2m+1)B(E_{m})(z)\\
        &\hspace{6cm}+(m+1)B(E_{m+1})(z))\\
        &=  \sum_{m\in \N_0} ((m+1)\lambda_{m+1} -(2m+1)\lambda_{m}+m\lambda_{m-1})B(E_{m})(z) \\
        &=\sum_{m\in \N_0}\mu_m B(E_{m})(z) 
    \end{align*}
    by reindexing. Also,
    \begin{align*}
        \mu_m&= m(\lambda_{m+1}-\lambda_{m})-m(\lambda_{m}-\lambda_{m-1})+(\lambda_{m+1}-\lambda_{m})\\
        &= m\Delta^2\lambda_{m-1}+\Delta\lambda_{m}.
    \end{align*}
    Thus if $\{\mu_{m}\}$ is a bounded sequence, then $S\in \domop$ and 
    $\Delta S= \sum_{m\in \N_0} \mu_m E_{m}.$ 
    Also if $S\in \domop$ then $\Delta B(S)=B(T)$ for some $T\in \bdd$. Then
    by the injectivity of the Berezin transform and equation \ref{eq:ptwise_Berezin}, 
    $$T=\sum_{m\in \N_0}\mu_{m} E_{m}.$$
    Thus $\{\mu_{m}\}$ is a bounded sequence, completing the proof.
\end{proof}

Now we characterize the sequence space $\dtwo$ as follows.

\begin{proposition}\label{prop:dtwo condition}
    A bounded sequence $\{x_m\}$ is in $\dtwo$ iff there exists $C>0$ s.t. 
    $$| m\Delta^2x_m|\leq C.$$
    for all $m\in \N_0$.
\end{proposition}

\begin{proof}
    By Proposition \ref{prop:dtwo}, a sequence $\{x_m\}$ is in $\dtwo$ iff 
    $$|m\Delta^2x_m+\Delta x_{m}|\leq C_1$$
for all $m\in\N_0$ for some $C_1>0$. However, since $\{\Delta x_m\}$ is bounded, the condition reduces to the claim.
\end{proof}

\subsection{Extension}

Given a sequence $\sigma \in \cbu$, one defines the extension $\lipext$ of $\sigma$ to $\R_+$ by the interpolation formula 
\begin{align}\label{eq:interpolation}
    \lipext(x)=\sigma(m)+(\sigma(m+1)-\sigma(m))\frac{\sqrt{x}-\sqrt{m}}{\sqrt{m+1}-\sqrt{m}}
\end{align}
for $x\in \R_+$ where $m=\lfloor x\rfloor$. Then $\lipext|_{\N_0}=\sigma$ and $\lipext$ is uniformly continuous with respect to the square-root metric on $\R_+$, and also $\|\lipext\|_\infty=\|\sigma\|_\infty$. 

Furthermore, we define $\lipextR$ on $\R$ by 
\begin{equation}
    \lipextR(x)=\lipext(x^2);\ \ \ x\in\R.
\end{equation}
Then $\lipextR \in \cbuR$ where $\lipextR$ denotes the usual space of bounded uniformly continuous functions on $\R$. Proofs of these facts are easy to verify and can be found in \cite{EM16}.

\begin{proposition}
    Let $f\in \cbuR$ s.t. $f$ has a bounded second derivative. Define the sequence $\sigma$ by $\sigma_n=f(n)$ for all $n\in \N_0$. Then $\sigma\in \dtwo$. 
\end{proposition}

\begin{proof}
    Let $n\in \N$. Then by the mean-value theorem for divided differences, there exists $x_n$ s.t.
    \begin{align*}
        \frac{f^{(2)}(x_n)}{2!}&=[f(\sqrt{n-1}), f(\sqrt{n}), f(\sqrt{n+1})]\\
        &=\frac{[f(\sqrt{n}), f(\sqrt{n+1})]-[f(\sqrt{n-1}), f(\sqrt{n})]}{\sqrt{n+1}-\sqrt{n-1}}\\
        &= \frac{(\sqrt{n+1}+\sqrt{n-1})}{2}\\
        &\hspace{2cm}\times\bigg( \frac{f(\sqrt{n+1})-f(\sqrt{n})}{\sqrt{n+1}-\sqrt{n}} -  \frac{f(\sqrt{n})-f(\sqrt{n-1})}{\sqrt{n}-\sqrt{n-1}} \bigg)\\
        &= \frac{(\sqrt{n+1}+\sqrt{n-1})}{2}\\
        &\hspace{.8cm}\times\big( (\sqrt{n+1}+\sqrt{n})(\sigma_{n+1}-\sigma_n)-(\sqrt{n}+\sqrt{n-1})(\sigma_n-\sigma_{n-1}) \big)\\
        &= \frac{1}{2}(\sqrt{n+1}+\sqrt{n-1})(\sqrt{n+1}+\sqrt{n})\Delta^2\sigma_{n-1}+ \Delta \sigma_{n-1}\\
        &\approx 2n\Delta^2\sigma_{n-1}+ \Delta \sigma_{n-1}
    \end{align*}
    Since $\{\Delta \sigma_{n-1}\}$ and  $\frac{f^{(2)}(x_n)}{2!}$ are bounded sequences it follows that $\{n\Delta^2\sigma_{n-1}\}$ is also bounded and hence $\sigma\in \dtwo$ by Proposition \ref{prop:dtwo condition}.
\end{proof}

\begin{theorem}\label{theo:dtwo closure}
    The sequence space $\dtwo$ is dense in $\cbu$, i.e. 
    $$\cbu=\overline{\dtwo}.$$
    Moreover the radial Toeplitz algebra $\toepalgrad$ is isomorphic to $\cbu$.
\end{theorem}

\begin{proof}
    Let $\sigma\in \cbu$ and let $\epsilon>0$. Then the $\lipextR\in \cbuR$. Since 
    $$\cbuR=\overline{\varphi\ast \cbuR},$$
    there exists $g\in \varphi\ast \cbuR$  s.t.
    $$\|\lipextR-g\|_\infty<\epsilon$$
    and $g$ has a bounded second derivative.
    Define the sequence $\nu$ by $\nu_n=g(n)$ for all $n\in \N_0$. Then
    $\nu\in \dtwo$ and 
    $$\|\sigma-\nu\|_\infty<\epsilon$$
    proving the theorem.
\end{proof}

The proof of the following proposition is similar to the proof of Proposition 6.2 in \cite{DO22}. However, we include it for completeness.

\begin{proposition}\label{prop:toep algebra n}
    We have that $\toepalgrad$ is independent of the dimension $n$, i.e. $$\mathfrak{T}_{rad}^{(n)}(L^\infty)\cong \mathfrak{T}_{rad}^{(1)}(L^\infty)\cong\cbu=\overline{\dtwo}.$$
\end{proposition}

\begin{proof}
    By Theorem \ref{theo:dtwo closure}, we have
    $$\mathfrak{T}_{rad}^{(1)}(L^\infty)\cong\cbu=\overline{\dtwo}.$$
    It is easy to verify that $\cbu$ is invariant under left and right shifts (see lemmas 3.2 and 3.3 in \cite{DO22}). Also, $\gamma_{n,a}$ is the $(n-1)^{th}$ left shift of $\gamma_{1,a}$, i.e. $\gamma_{n,a}=\tau_L^{n-1}(\gamma_{1,a})$.
    Let $\sigma\in \cbu$ and let $\epsilon>0$. Then $\sigma':=\tau_R^{n-1}(\sigma)\in \cbu$. Also 
    $$\cbu=\overline{\{\gamma_{1,a}\mid a\in \bddf\}},$$
    there exists $a\in \bddf$ s.t.
    $$\|\sigma'-\gamma_{1,a}\|_\infty<\epsilon.$$
    Then 
    \begin{align*}
        \|\sigma-\gamma_{n,a}\|_\infty &= \|\tau_L^{n-1}(\sigma')-\tau_L^{n-1}(\gamma_{1,a})\|_\infty\\
        &\leq \|\tau_L^{n-1}\| \|\sigma'-\gamma_{1,a}\|_\infty\\
        &<\epsilon.
    \end{align*}
\end{proof}

\bibliographystyle{plain} 
\bibliography{Laplacian}

\begin{thebibliography}{10}

\bibitem{BCI10}
W.~Bauer, L.~A. Coburn, and J.~Isralowitz.
\newblock Heat flow, {BMO}, and the compactness of {T}oeplitz operators.
\newblock {\em J. Funct. Anal.}, 259(1):57--78, 2010.

\bibitem{BHV14b}
W.~Bauer, C.~Herrera Ya\~nez, and N.~Vasilevski.
\newblock Eigenvalue characterization of radial operators on weighted {B}ergman spaces over the unit ball.
\newblock {\em Integral Equations Operator Theory}, 78(2):271--300, 2014.

\bibitem{BHV14}
W.~Bauer, C.~Herrera Ya\~nez, and N.~Vasilevski.
\newblock {$(m,\lambda)$}-{B}erezin transform and approximation of operators on weighted {B}ergman spaces over the unit ball.
\newblock In {\em Operator theory in harmonic and non-commutative analysis}, volume 240 of {\em Oper. Theory Adv. Appl.}, pages 45--68. Birkh\"auser/Springer, Cham, 2014.

\bibitem{BR22}
W.~Bauer and M.~A. Rodriguez~Rodriguez.
\newblock Commutative toeplitz algebras and their gelfand theory: old and new results.
\newblock {\em Complex Anal. Oper. Theory}, 16(6):Paper No. 77, 37, 2022.

\bibitem{BBLS22}
E.~Berge, S.M. Berge, F.~Luef, and E.~Skrettingland.
\newblock Affine quantum harmonic analysis.
\newblock {\em J. Funct. Anal.}, 282(4):Paper No. 109327, 64, 2022.

\bibitem{BC94}
C.~A. Berger and L.~A. Coburn.
\newblock Heat flow and {B}erezin-{T}oeplitz estimates.
\newblock {\em Amer. J. Math.}, 116(3):563--590, 1994.

\bibitem{DOQ18}
M.~Dawson, G.~\'{O}lafsson, and R.~Quiroga-Barranco.
\newblock The restriction principle and commuting families of {T}oeplitz operators on the unit ball.
\newblock {\em S\~{a}o Paulo J. Math. Sci.}, 12(2):196--226, 2018.

\bibitem{DM23}
V.~Dewage and M.~Mitkovski.
\newblock Density of {T}oeplitz operators in rotation-invariant {T}oeplitz algebras, 2023.
\newblock arXiv:2310.12367.

\bibitem{DO22}
V.~Dewage and G.~\'{O}lafsson.
\newblock Toeplitz operators on the {F}ock space with quasi-radial symbols.
\newblock {\em Complex Anal. Oper. Theory}, 16(4):Paper No. 61, 32, 2022.

\bibitem{EM16}
K.~Esmeral and E.~A. Maximenko.
\newblock Radial {T}oeplitz operators on the {F}ock space and square-root-slowly oscillating sequences.
\newblock {\em Complex Anal. Oper. Theory}, 10(7):1655--1677, 2016.

\bibitem{EV16}
K.~Esmeral and N.~Vasilevski.
\newblock ${C}^*$-algebra generated by horizontal {T}oeplitz operators on the {F}ock space.
\newblock {\em Bol. Soc. Mat. Mex. (3)}, 22(2):567--582, 2016.

\bibitem{FH23}
F.~Fulsche and R.~Hagger.
\newblock Quantum harmonic analysis for polyanalytic {F}ock spaces, 2023.
\newblock arXiv:2308.11292.

\bibitem{F19}
R.~Fulsche.
\newblock Correspondence theory on {$p$}-{F}ock spaces with applications to {T}oeplitz algebras.
\newblock {\em J. Funct. Anal.}, 279(7):108661, 41, 2020.

\bibitem{FR23}
R.~Fulsche and M.~A. Rodriguez~Rodriguez.
\newblock Commutative {G}-invariant toeplitz ${C}^*$-algebras on the {F}ock space and their {G}elfand theory through quantum harmonic analysis.
\newblock 2023.
\newblock arXiv:2307.15632.

\bibitem{GKV03}
S.~Grudsky, A.~Karapetyants, and N.~Vasilevski.
\newblock Toeplitz operators on the unit ball in {${\mathbb{C}}^n$} with radial symbols.
\newblock {\em J. Operator Theory}, 49(2):325--346, 2003.

\bibitem{GMV13}
S.~M. Grudsky, E.~A. Maximenko, and N.~L. Vasilevski.
\newblock Radial {T}oeplitz operators on the unit ball and slowly oscillating sequences.
\newblock {\em Commun. Math. Anal.}, 14(2):77--94, 2013.

\bibitem{GV02}
S.~M. Grudsky and N.~L. Vasilevski.
\newblock Toeplitz operators on the {F}ock space: radial component effects.
\newblock {\em Integral Equations Operator Theory}, 44(1):10--37, 2002.

\bibitem{H23}
S.~Halvdansson.
\newblock Quantum harmonic analysis on locally compact groups.
\newblock {\em Journal of Functional Analysis}, 285(8):110096, 2023.

\bibitem{KL21}
J.~Keller and F.~Luef.
\newblock Polyanalytic {T}oeplitz operators: isomorphisms, symbolic calculus and approximation of {W}eyl operators.
\newblock {\em J. Fourier Anal. Appl.}, 27(3):Paper No. 47, 34, 2021.

\bibitem{KLSW12}
J.~Kiukas, P.~Lahti, J.~Schultz, and R.~F. Werner.
\newblock Characterization of informational completeness for covariant phase space observables.
\newblock {\em J. Math. Phys.}, 53(10):102103, 11, 2012.

\bibitem{LS18}
F.~Luef and E.~Skrettingland.
\newblock Convolutions for localization operators.
\newblock {\em J. Math. Pures Appl. (9)}, 118:288--316, 2018.

\bibitem{LS20}
F.~Luef and E.~Skrettingland.
\newblock A {W}iener {T}auberian theorem for operators and functions.
\newblock {\em J. Funct. Anal.}, 280(6):Paper No. 108883, 44, 2021.

\bibitem{S04}
D.~Su\'{a}rez.
\newblock Approximation and symbolic calculus for {T}oeplitz algebras on the {B}ergman space.
\newblock {\em Rev. Mat. Iberoamericana}, 20(2):563--610, 2004.

\bibitem{S05}
D.~Su\'{a}rez.
\newblock Approximation and the {$n$}-{B}erezin transform of operators on the {B}ergman space.
\newblock {\em J. Reine Angew. Math.}, 581:175--192, 2005.

\bibitem{S08}
D.~Su\'{a}rez.
\newblock The eigenvalues of limits of radial {T}oeplitz operators.
\newblock {\em Bull. Lond. Math. Soc.}, 40(4):631--641, 2008.

\bibitem{W84}
R.~Werner.
\newblock Quantum harmonic analysis on phase space.
\newblock {\em J. Math. Phys.}, 25(5):1404--1411, 1984.

\end{thebibliography}

\end{document}